\documentclass[10pt]{amsart}
\usepackage[margin=1.5in]{geometry}               

\usepackage{enumerate}
\usepackage{amssymb}
\usepackage{amsmath}
\usepackage{amscd}
\usepackage{amsthm}
\usepackage{amsfonts}
\usepackage{graphicx}
\usepackage[all]{xy}
\usepackage{verbatim}
\usepackage{hyperref}
\usepackage{epstopdf}
\usepackage[utf8]{inputenc}
\usepackage{bbm}
\usepackage{esint}
\usepackage{tikz}

\theoremstyle{definition}\newtheorem*{claim}{Claim}
\newtheorem*{rmk}{Remark}

\theoremstyle{definition}

\theoremstyle{definition}

\theoremstyle{definition}
\theoremstyle{definition}
\theoremstyle{definition}
\numberwithin{theorem}{section}

\newcommand{\bR}{\mathbb{R}}
\newcommand{\bZ}{\mathbb{Z}}

\newcommand{\bT}{\mathbb{T}}

\newcommand\pa[1]{\left[#1\right]}

\newcommand\on[1]{\operatorname{#1}} 

\newcommand\smallmat[1]{\pa{\begin{smallmatrix}#1\end{smallmatrix}}}

\newcommand{\onto}{\xymatrix{\ar@{>>}[r]&}}
\newcommand{\da}[4]{\xymatrix{#1 \ar@<.5ex>[r]^{#2} \ar@<-.5ex>[r]_{#3} & #4}}

\newif\ifdraft\drafttrue

\theoremstyle{definition}

\title{A badly expanding set on the $2$-torus}
\date{\today}
\author{Rene R\"uhr}

\begin{document}
\maketitle
\begin{abstract}
We give a counterexample to a conjecture stated in \cite{liniallondon} regarding expansion on $\bT^2$ under $\smallmat{1&1\\ 0&1}$ and $\smallmat{1&0\\ 1&1}$.
\end{abstract}

Let $\Sigma_{\on{D}}$ be the set containing the linear transformation $\sigma_1=\smallmat{1&1\\ 0&1}$ and its transpose $\sigma_2=\smallmat{1&0\\ 1&1}$, and let $\Sigma_{\on{U}}=\Sigma_{\on{D}}\cup \Sigma_{\on{D}}^{-1}$, adding the inverses of $\sigma_1$ and $\sigma_2$. Using these transformations, Linial and London \cite{liniallondon} studied an infinite $4$-regular expander graph, showing the following expansion property: For any bounded measurable set $A\subset \bR^2$ of the plane, one has
\begin{equation}
\label{bounds}
m\left(A\cup\bigcup_{\sigma\in \Sigma_{\on{U}}}\sigma(A)\right)\geq 2m\left(A\right)
\text{ and }
m\left(A\cup\bigcup_{\sigma\in \Sigma_{\on{D}}}\sigma(A)\right)\geq \frac{4}{3}m\left(A\right)
\end{equation}
where $m$ denotes the Lebesgue measure of a set and the bounds are sharp. 
Note that $\Sigma_{\on{U}}\subset \on{SL}_2(\bZ)$ and thus its elements also act on $\bT^2=\bR^2/\bZ^2$, and this action is measure preserving with respect to the induced probability measure $m_{\bT^2}$ on $\bT^2$.
It was conjectured in \cite{liniallondon} and in \cite{expandersurvey}[Conjecture~4.5] that there is a constant $c>0$ such that for $A\subset \bT^2=\bR^2/\bZ^2$ with $m_{\bT^2}(A)\leq c$  the estimate of line~(\ref{bounds}) with $m_{\bT^2}$ in place of $m$ holds.
Below we give a simple counterexample to this conjecture.
Let $\pi:\bR^2\to\bT^2$ denote the natural projection map. Let $\varepsilon>0$ and define
\[C_U=\pi(\{(x,y)\in\bR^2: |x|\leq\varepsilon \text{ or } |y|\leq\varepsilon\}).\]
and
\[
C_D=C_U\cup \pi(\{(x,y)\in\bR^2: |x+y|\leq \varepsilon\}).
\]
These sets are of arbitrary small measure as $\varepsilon\to 0$ and satisfy
\begin{claim}
$m_{\bT^2}\left(C_U\cup\bigcup_{\sigma\in \Sigma_{\on{U}}}\sigma(C_U)\right)<2m_{\bT^2}(C_U)$
and
$m_{\bT^2}\left(C_D\cup\bigcup_{\sigma\in \Sigma_{\on{U}}}\sigma(C_D)\right)<\frac43m_{\bT^2}(C_U)$.
\end{claim}
\begin{proof}
The following picture depicts the set $C_U$ in red and the image under $\Sigma_{\on{U}}$ in blue. 
We note that the overlapping triangles outside the square are to be seen modulo $1$, thus wrap up and do not amount to additional mass.

\medskip 
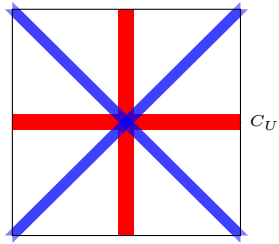
\begin{center}
\begin{tikzpicture}

  \filldraw[fill=red, draw=none] (0,1.4) rectangle (3,1.6);
  \filldraw[fill=red, draw=none] (1.4,0) rectangle (1.6,3);

 \coordinate (a) at (3,2.9);
\coordinate (b) at (0,-0.1);
\coordinate (c) at (0,0.1);
\coordinate (d) at (3,3.1);
\draw[draw=blue,opacity=0,fill=blue, fill opacity=0.5] (a) -- (b) -- (c) -- (d) -- cycle;

  \coordinate (a) at (-0.1,0);
\coordinate (b) at (0.1,0);
\coordinate (c) at (3.1,3);
\coordinate (d) at (2.9,3);
\draw[draw=blue,opacity=0,fill=blue, fill opacity=0.5] (a) -- (b) -- (c) -- (d) -- cycle;

  \coordinate (a) at (0,3.1);
\coordinate (b) at (3,0.1);
\coordinate (c) at (3,-0.1);
\coordinate (d) at (0,2.9);
\draw[draw=blue,opacity=0,fill=blue, fill opacity=0.5] (a) -- (b) -- (c) -- (d) -- cycle;

 \coordinate (a) at (-0.1,3);
\coordinate (b) at (0.1,3);
\coordinate (c) at (3.1,0);
\coordinate (d) at (2.9,0);
\draw[draw=blue,opacity=0,fill=blue, fill opacity=0.5] (a) -- (b) -- (c) -- (d) -- cycle;

\coordinate (a) at (0,0);
\coordinate (b) at (0,3);
\coordinate (c) at (3,3);
\coordinate (d) at (3,0);
\draw[black,thin] (a) -- (b) -- (c) -- (d) -- cycle;

  \filldraw (3,1.5) circle (0pt) node[align=left,   right] {\tiny\(C_U\)};
\end{tikzpicture}
\end{center}

The area of $C_U$ is $4\varepsilon-(2\varepsilon)^2$. The set $\bigcup_{\sigma\in \Sigma_{\on{U}}}\sigma(C_U)$ overlaps $C_U$ in the square of area $(2\varepsilon)^2$ and the $8$ adjacent triangles each of area $\varepsilon^2/2$. At the corners of the torus, we have additional $4$ overlapping triangles of the same size (note that we only have to remove half of the $8$ drawn triangles to match the fact that we are subtracting from the area of $2$ $2\varepsilon$-thick (sheared) rectangles, and not $4$). This amounts to a total area not contained in $C_U$ of $4\varepsilon-(2\varepsilon)^2-4\varepsilon^2-2\varepsilon^2=4\varepsilon-10\varepsilon^2<4\varepsilon-4\varepsilon^2$.

In the following picture of $C_D$, we did not draw the image of $\pi(\{|x+y|\leq \varepsilon\})$ as it is mapped to $C_U$ under $\Sigma_{\on{D}}$.

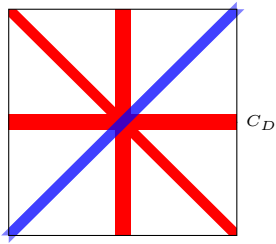
\begin{center}
\begin{tikzpicture}

  \filldraw[fill=red, draw=none] (0,1.4) rectangle (3,1.6);
  \filldraw[fill=red, draw=none] (1.4,0) rectangle (1.6,3);

\coordinate (a) at (0,3);
\coordinate (b) at (0.1,3);
\coordinate (c) at (3,0.1);
\coordinate (d) at (3,0);
\coordinate (f) at (2.9,0);
\coordinate (g) at (0,2.9);
\filldraw[fill=red, draw=none] (a) -- (b) -- (c) -- (d) -- (f) -- (g) -- cycle;

 \coordinate (a) at (3,2.9);
\coordinate (b) at (0,-0.1);
\coordinate (c) at (0,0.1);
\coordinate (d) at (3,3.1);
\draw[draw=blue,opacity=0,fill=blue, fill opacity=0.5] (a) -- (b) -- (c) -- (d) -- cycle;

  \coordinate (a) at (-0.1,0);
\coordinate (b) at (0.1,0);
\coordinate (c) at (3.1,3);
\coordinate (d) at (2.9,3);
\draw[draw=blue,opacity=0,fill=blue, fill opacity=0.5] (a) -- (b) -- (c) -- (d) -- cycle;

\coordinate (a) at (0,0);
\coordinate (b) at (0,3);
\coordinate (c) at (3,3);
\coordinate (d) at (3,0);
\draw[black,thin] (a) -- (b) -- (c) -- (d) -- cycle;

  \filldraw (3,1.5) circle (0pt) node[align=left,   right] {\tiny\(C_D\)};

\end{tikzpicture}
\end{center}

We see that $m_{\bT^2}(C_D)=6\varepsilon-10\varepsilon^2$, as the diagonal overlaps of area $(2\varepsilon)^2+\varepsilon^2$ with $C_U$ and has two triangles each of area $\varepsilon^2/2$ cut off at the corners of the torus.

The image consists of the single other diagonal and accounts for $2\varepsilon-((2\varepsilon)^2+\varepsilon^2)-\varepsilon^2=2\varepsilon-6\varepsilon^2$ mass, which is less than $\frac13m_{\bT^2}(C_D)=2\varepsilon-\frac{10}{3}\varepsilon^2$.

\end{proof}

\begin{rmk}
We note that it is possible to shift $C_U$ by $(1/2,1/2)$ to have a set with support bounded away from $(0,0)+\bZ^2$ and still badly expanding.
\end{rmk}

\medskip 
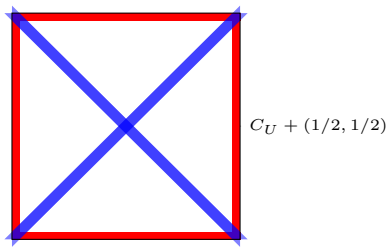
\begin{center}
\begin{tikzpicture}

\filldraw[fill=red, draw=none] (0,0) rectangle (3,0.1);
\filldraw[fill=red, draw=none] (0,2.9) rectangle (3,3);
\filldraw[fill=red, draw=none] (2.9,0) rectangle (3,3);
\filldraw[fill=red, draw=none] (0,0) rectangle (0.1,3);

\coordinate (a) at (3,2.9);
\coordinate (b) at (0,-0.1);
\coordinate (c) at (0,0.1);
\coordinate (d) at (3,3.1);
\draw[draw=blue,opacity=0,fill=blue, fill opacity=0.5] (a) -- (b) -- (c) -- (d) -- cycle;

\coordinate (a) at (-0.1,0);
\coordinate (b) at (0.1,0);
\coordinate (c) at (3.1,3);
\coordinate (d) at (2.9,3);
\draw[draw=blue,opacity=0,fill=blue, fill opacity=0.5] (a) -- (b) -- (c) -- (d) -- cycle;

  \coordinate (a) at (0,3.1);
\coordinate (b) at (3,0.1);
\coordinate (c) at (3,-0.1);
\coordinate (d) at (0,2.9);
\draw[draw=blue,opacity=0,fill=blue, fill opacity=0.5] (a) -- (b) -- (c) -- (d) -- cycle;

 \coordinate (a) at (-0.1,3);
\coordinate (b) at (0.1,3);
\coordinate (c) at (3.1,0);
\coordinate (d) at (2.9,0);
\draw[draw=blue,opacity=0,fill=blue, fill opacity=0.5] (a) -- (b) -- (c) -- (d) -- cycle;

\coordinate (a) at (0,0);
\coordinate (b) at (0,3);
\coordinate (c) at (3,3);
\coordinate (d) at (3,0);
\draw[black,thin] (a) -- (b) -- (c) -- (d) -- cycle;

  \filldraw (3,1.5) circle (0pt) node[align=left,   right] {\tiny\(C_U+(1/2,1/2)\)};
\end{tikzpicture}
\end{center}

\begin{rmk}
The expanding properties on $\bR^2$ and $\bT^2$ in \cite{liniallondon} are considered to be the continuous analogue of the expander family $\bZ^2/(p\bZ)^2-\{(0,0)\}$ with adjacency relations given by $\Sigma_U$ for $p$ prime. These are indeed expanders by Selberg's $3/16$-theorem on the spectral gap of congruence subgroups in $\on{SL}_2(\bZ)$ (see e.g.\ the first bullet point after \cite{expandersurvey}[Proposition~11.17]). 
One can of course formulate the analogous question of line (\ref{bounds}), and ask whether for a sufficently small subset $A$ of $\bZ^2/(p\bZ)^2-\{(0,0)\}$ one has $\left|A\cup\bigcup_{\sigma\in \Sigma_{\on{U}}}\sigma(A)\right|\geq 2|A|$ (and similarly in the case of the directed graph using relations $\Sigma_D$). This translates to the vertex isoperimetric parameter (see \cite{expandersurvey}[Section~4.6]). The set $(\bZ\times\{0\}\cup\{0\}\times\bZ)/(p\bZ)^2-\{(0,0)\}$ satisfies the bound with equality (and we have the analogous statement for the directed graph).
\end{rmk}
\bibliographystyle{alpha}

\end{document}